\documentclass[11pt]{article}

\usepackage[utf8]{inputenc} 
\usepackage{amsmath}
\usepackage{graphicx}
\usepackage{subcaption}
\usepackage[top=30mm, bottom=30mm, left=27mm, right=27mm]{geometry}
\usepackage{amsfonts}
\usepackage{mathtools}
\usepackage{subcaption}
\usepackage{amssymb}
\usepackage{amsthm}
\usepackage{chngpage}
\usepackage{hyperref}
\hypersetup{
	colorlinks=true,
	linkcolor=black,
	citecolor=black,
	filecolor=black,      
	urlcolor=black,
}
\usepackage{enumerate}
\usepackage{makeidx}
\usepackage{bm}
\usepackage{thmtools}
\usepackage{thm-restate}

\makeindex

\makeatletter
\renewenvironment{proof}[1][\proofname] {\par\pushQED{\qed}\normalfont\topsep6\p@\@plus6\p@\relax\trivlist\item[\hskip\labelsep\bfseries#1\@addpunct{.}]\ignorespaces}{\popQED\endtrivlist\@endpefalse}
\makeatother

\newcommand{\PP}{\mathcal{P}}
\renewcommand{\AA}{\mathcal{A}}
\newcommand{\BB}{\mathcal{B}}

\newtheorem{proposition}{Proposition}
\newtheorem{lemma}[proposition]{Lemma}
\newtheorem{theorem}[proposition]{Theorem}

\usepackage{amsthm}

\theoremstyle{definition}

\newtheorem*{remark*}{Remark}
\newtheorem*{theorem*}{Theorem}

\mathcode`l="8000
\begingroup
\makeatletter
\lccode`\~=`\l
\DeclareMathSymbol{\lsb@l}{\mathalpha}{letters}{`l}
\lowercase{\gdef~{\ifnum\the\mathgroup=\m@ne \ell \else \lsb@l \fi}}%
\endgroup

\title{A note on the orientation covering number}
\author{Barnabás Janzer\thanks{Department of Pure Mathematics and Mathematical Statistics, University of Cambridge, Wilberforce Road, Cambridge CB3 0WB, United Kingdom. Email: bkj21@cam.ac.uk. This work was supported by EPSRC DTG.}}
\date{\vspace{-21pt}}

\begin{document}
	\maketitle
	
\begin{abstract}
Given a graph $G$, its orientation covering number $\sigma(G)$ is the smallest non-negative integer $k$ with the property that we can choose $k$ orientations of $G$ such that whenever $x, y, z$ are vertices of $G$ with $xy,xz\in E(G)$ then there is a chosen orientation in which both $xy$ and $xz$ are oriented away from $x$. Esperet, Gimbel and King showed that $\sigma(G)\leq \sigma\left(K_{\chi(G)}\right)$, where $\chi(G)$ is the chromatic number of $G$, and asked whether we always have equality. In this note we prove that it is indeed always the case that $\sigma(G)=\sigma(K_{\chi(G)})$. We also determine the exact value of $\sigma(K_n)$ explicitly for `most' values of $n$.
\end{abstract}

\section{Introduction}

Given a non-empty graph $G$ and $k$ orientations $\vec{G}_1,\dots,\vec{G}_k$ of $G$, we say that $\vec{G}_1,\dots,\vec{G}_k$ is an \textit{orientation covering} of $G$ if whenever $x,y,z\in V(G)$ with $xy,xz\in E(G)$ then there is an orientation in which both $xy$ and $xz$ are oriented away from $x$ (i.e., there is some $i$ such that $(x,y),(x,z)\in E(\vec{G}_i)$). The \textit{orientation covering number} $\sigma(G)$ of $G$ is the smallest positive integer $k$ such that there is a list of $k$ orientations forming an orientation covering of $G$. Orientation coverings were introduced by Esperet, Gimbel and King \cite{esperet2010covering}, who used them to study the minimal number of equivalence subgraphs needed to cover a given graph.

Esperet, Gimbel and King \cite{esperet2010covering} showed that $\sigma(G)\leq \sigma\left(K_{\chi(G)}\right)$ for any graph $G$, where $\chi$ denotes the chromatic number. They asked whether we always have $\sigma(G)=\sigma\left(K_{\chi(G)}\right)$. In this note we answer this question in the positive.

\begin{theorem}\label{thm_orcovequal}
	For any non-empty graph $G$, we have $\sigma(G)=\sigma\left(K_{\chi(G)}\right)$.
\end{theorem}


The value of $\sigma(K_n)$ has been investigated by Esperet, Gimbel and King \cite{esperet2010covering}, who determined its order of magnitude and the exact values for small values of $n$. An observation of Gyárfás (see \cite{esperet2010covering}) shows that we have $\chi(DS_n)\leq \sigma(K_n)\leq \chi(DS_n)+2$, where $DS_n$ is the double-shift graph on $n$ vertices. Using the results of Füredi, Hajnal, Rödl and Trotter \cite{furedi1992interval} on the chromatic number of $DS_n$, this gives $\sigma(K_n)=\log\log n+\frac{1}{2}\log\log\log n+O(1)$. (All logarithms in this paper are base $2$.) In this note we will also determine the value of $\sigma(K_n)$ exactly in terms of a certain sequence of positive integers sometimes called the Ho\c{s}ten--Morris numbers. As a corollary, we get the following improved estimate.

\begin{theorem}\label{thm_orcovestimate}
	We have $\sigma(K_n)=\lceil\log\log n+\frac{1}{2}\log\log\log n+\frac{1}{2}(\log \pi+1)+o(1)\rceil$ as $n\to \infty$.
\end{theorem}

Given a positive integer $k$, let $[k]$ denote $\{1,\dots,k\}$, as usual.
Given a family $\AA\subseteq \PP([k])$ of subsets of $[k]$, we say that $\AA$ is \textit{intersecting} if whenever $S,T\in \AA$ then $S\cap T\not =\emptyset$. We say that $\AA$ is \textit{maximal intersecting} if $\AA$ is intersecting and whenever $\BB\supseteq\AA$ and $\BB$ is intersecting then $\BB=\AA$. (Equivalently, if $\AA$ is intersecting and $|\AA|=2^{k-1}$.) The following characterisation of $\sigma(G)$ is the key to our results.

\begin{theorem}\label{theorem_orientationcover}
	For any non-empty graph $G$, $\sigma(G)$ is the smallest positive integer $k$ such that there are at least $\chi(G)$ maximal intersecting families over $[k]$.
\end{theorem}

Clearly, Theorem~\ref{theorem_orientationcover} implies Theorem~\ref{thm_orcovequal}. Let $\lambda(k)$ denote the number of maximal intersecting families over $[k]$. The numbers $\lambda(k)$ are sometimes called Ho\c{s}ten--Morris numbers, after a paper of Ho\c{s}ten and Morris \cite{hocsten1999order} in which they showed that the order dimension of $K_n$ is the smallest positive integer $k$ with $\lambda(k)\geq n$. An equivalent formulation of their result is that the minimal number of linear orders on $[n]$ with the property that the induced orientations of $K_n$ form an orientation covering is the smallest positive integer $k$ with $\lambda(k)\geq n$. Note that by Theorem \ref{theorem_orientationcover} this number is the same as the orientation covering number of $K_n$.

Although no exact or asymptotic formula is known for $\lambda(k)$, it was shown by Brouwer, Mills, Mills and Verbeek \cite{brouwer2013counting} that
\begin{equation}\label{eq_HMbound}
\log \lambda(k)\sim \frac{2^k}{\sqrt{2\pi k}}.
\end{equation}
Furthermore, the exact values of $\lambda(k)$ are known \cite{brouwer2013counting} for $k$ up to 9, with $\lambda(9)\approx 4\times 10^{20}$.

Theorem~\ref{thm_orcovestimate} follows from Theorem~\ref{theorem_orientationcover} and \eqref{eq_HMbound}. Indeed, taking logarithms in \eqref{eq_HMbound} shows that $\sigma(K_n)$ is the smallest positive integer $k$ with $\log\log n\leq k-\frac{1}{2}(\log\pi+1)-\frac{1}{2}\log k+o(1)$, which gives $\sigma(K_n)=\lceil\log\log n+\frac{1}{2}\log\log\log n+\frac{1}{2}(\log \pi+1)+o(1)\rceil$.

\section{Proof of Theorem \texorpdfstring{\ref{theorem_orientationcover}}{3}}

The proof is based on the following observation.

\begin{lemma}\label{lemma_altcharacterisation}
	For any non-empty graph $G$, $\sigma(G)$ is the smallest positive integer $k$ with the property that there is a collection $(\AA_v)_{v\in V(G)}$ of subsets of $\PP([k])$ (i.e., $\AA_v\subseteq \PP([k])$ for all $v$) such that the following two conditions hold.
	\begin{enumerate}
		\item If $uv\in E(G)$, then there exists $S\in \AA_u$ and $T\in \AA_v$ such that $S\cap T=\emptyset$.
		\item For all $v\in V(G)$ and $S,T\in \AA_v$, we have $S\cap T\not =\emptyset$. (I.e., $\AA_v$ is intersecting.)
	\end{enumerate}
\end{lemma}
\begin{proof}
	First assume that $\sigma(G)=k$ and $\vec{G}_1,\dots,\vec{G}_k$ form an orientation cover of $G$. For each directed edge $(x,y)$ of $G$, let $S_{(x,y)}=\{i\in [k]: (x,y)\in E(\vec{G}_i)\}$. Let $\AA_v=\{S_{(v,w)}:vw\in E(G)\}$. Clearly $S_{(v,w)}\cap S_{(w,v)}=\emptyset$, so Condition 1 holds. Also, we have $S_{(v,w)}\cap S_{(v,w')}\not=\emptyset$ whenever $vw,vw'\in E(G)$, since by assumption there is an $i$ such that $(v,w),(v,w')\in E(\vec{G}_i)$. So Condition 2 holds as well.
	
	Conversely, suppose that we have such a collection $(\AA_v)_{v\in V(G)}$ with $\AA_v\subseteq \PP([k])$ for all $v$. For each $uv\in E(G)$, pick $S_{(u,v)}\in \AA_u$ and $S_{(v,u)}\in \AA_v$ such that $S_{(u,v)}\cap S_{(v,u)}=\emptyset$. Define the orientations $\vec{G}_1,\dots,\vec{G}_k$ of $G$ by orienting the edge $uv$ from $u$ to $v$ in $\vec{G}_i$ if $i\in S_{(u,v)}$, from $v$ to $u$ if $i\in S_{(v,u)}$, and arbitrarily otherwise. This is clearly well-defined, and whenever $uv,uw\in E(G)$, then $S_{(u,v)}\cap S_{(u,w)}\not =\emptyset$ (by Condition 2). This gives $\sigma(G)\leq k$, as claimed.
\end{proof}
\begin{proof}[Proof of Theorem \ref{theorem_orientationcover}]
We first show the lower bound for $\sigma(G)$. Let $G$ be any non-empty graph, and let $(\AA_v)_{v\in V(G)}$ be as in Lemma \ref{lemma_altcharacterisation} for $k=\sigma(G)$. For each $v\in V(G)$, let $\BB_v$ be a maximal intersecting family with $\BB_v\supseteq \AA_v$. Note that the families $(\BB_v)_{v\in V(G)}$ still satisfy both conditions in Lemma \ref{lemma_altcharacterisation}. Furthermore, $v\mapsto \BB_v$ is a proper vertex-colouring (since each $\BB_v$ is intersecting but $\BB_v\cup\BB_w$ is not whenever $vw\in E(G)$). It follows that the number of maximal intersecting families over $[k]$ is at least $\chi(G)$.

Conversely, assume that $k$ is a positive integer such that there are at least $\chi(G)$ distinct maximal intersecting families $\BB_1,\dots,\BB_k$ over $[k]$. Let $c:V(G)\mapsto [\chi(G)]$ be a proper vertex-colouring of $G$, and set $\AA_v=\BB_{c(v)}$ for each $v$. Certainly each $\AA_v$ is intersecting. Furthermore, by maximality, no $\AA_v\cup\AA_w$ can be intersecting when $c(v)\not=c(w)$, and hence $\AA_v\cup \AA_w$ is not intersecting when $vw\in E(G)$. It follows that $(\AA_v)_{v\in V(G)}$ satisfies both conditions in Lemma \ref{lemma_altcharacterisation} and so $\sigma(G)\leq k$. 
\end{proof}

\bibliography{Bibliography}
\bibliographystyle{abbrv}

\end{document}